\documentclass{article}

\usepackage[utf8]{inputenc}

\usepackage{amsfonts}
\usepackage{amssymb}
\usepackage{amsthm}
\usepackage[url=false,giveninits=true,isbn=false,maxbibnames=99,sortcites]{biblatex}
\usepackage{bm}
\usepackage{booktabs}
\usepackage{caption}
\usepackage{centernot}
\usepackage{enumerate}
\usepackage{float}
\usepackage[a4paper]{geometry}
\usepackage{graphicx}
\usepackage{lmodern}
\usepackage{mathtools}
\usepackage{microtype}
\usepackage{multicol}
\usepackage{multirow}
\usepackage{tikz}
\usepackage{xcolor}

\usetikzlibrary{calc}
\usetikzlibrary{cd}
\usetikzlibrary{matrix}

\tikzset{ shorten <>/.style={ shorten >=#1, shorten <=#1}}

\restylefloat{table}

\usepackage[colorlinks=true,linkcolor=red!50!black,citecolor=green!50!black,urlcolor=blue!50!black,hyperindex,breaklinks]{hyperref}	
\usepackage[capitalise]{cleveref}

\renewbibmacro{in:}{}
\DeclareFieldFormat[article]{title}{#1}

\newtheoremstyle{break}
	{}
	{}
	{}
	{}
	{\bfseries}
	{}
	{\newline}
	{}
\newtheoremstyle{cite}
	{}
	{}
	{}
	{}
	{\bfseries}
	{}
	{\newline}
	{\thmname{#1} \thmnumber{#2} \thmnote{#3}}
\theoremstyle{break}

\newtheorem{lemma}[subsection]{Lemma}

\newtheorem{defn}[subsection]{Definition}
\newtheorem{propn}[subsection]{Proposition}
\newtheorem*{notation*}{Notation}
\newtheorem{introthm}{Theorem}

\theoremstyle{cite}

\newtheorem{lemmacite}[subsection]{Lemma}

\crefname{defn}{Definition}{Definitions}
\crefname{propn}{Proposition}{Propositions}

\let\H\relax

\renewcommand{\epsilon}{\varepsilon}

\newcommand{\bigmoduleshape}[2]{\begin{array}{c}	#1 \\ #2	\end{array}}
\newcommand{\twomoduleshape}{\bigmoduleshape}
\newcommand{\threemoduleshape}[3]{\begin{array}{c}	#1 \\	#2 \\	#3 \end{array}}
\newcommand{\fourmoduleshape}[4]{\begin{array}{c}	#1 \\	#2 \\	#3 \\	#4 \end{array}}

\newcommand{\perm}{\mathcal{L}}

\newcommand{\surj}{\twoheadrightarrow}
\newcommand{\twist}[1]{{}^{#1}\!}
\newcommand{\upth}{\textsuperscript{th}}

\DeclareMathOperator{\Ext}{Ext}
\DeclareMathOperator{\H}{H}
\DeclareMathOperator{\head}{head}
\DeclareMathOperator{\heart}{\mathcal{H}}
\DeclareMathOperator{\Hom}{Hom}

\DeclareMathOperator{\Irr}{Irr}
\DeclareMathOperator{\PC}{\mathcal{P}}
\DeclareMathOperator{\rad}{rad}

\DeclareMathOperator{\soc}{soc}

\DeclareMathOperator{\PGammaL}{P\Gamma L}
\DeclareMathOperator{\PGL}{PGL}

\DeclareMathOperator{\PSL}{PSL}
\DeclareMathOperator{\PSU}{PSU}
\DeclareMathOperator{\SL}{SL}

\DeclareMathOperator{\Sz}{Sz}
\DeclareMathOperator{\SzB}{{}^2 B_2}
\DeclareMathOperator{\Ree}{{}^2G_2}

\DeclarePairedDelimiter{\abs}{\lvert}{\rvert}

\begin{document}

\begin{center}
{\Large Cohomology and Ext for rank one finite groups of Lie type in cross characteristic}

{\large Jack Saunders\footnotemark} \footnotetext{A portion of this work was completed during my PhD, thus I would like to thank my supervisors Corneliu Hoffman \& Chris Parker for everything they have done for me so far. I am grateful to the LMS for their financial support via the grant ECF-1920-30 and to Gunter Malle for his support during the corresponding fellowship. I was supported by Australian Research Council Discovery Project Grant DP190101024 while the work for this paper was undertaken.}

Department of Mathematics and Statistics, The University of Western Australia\\
35 Stirling Highway, Perth, WA 6009, Australia\vskip-0.8em
\href{mailto:jack-saunders@hotmail.co.uk}{jack-saunders@hotmail.co.uk}
\end{center}

\begin{abstract}
	We compute the dimensions of \(\Ext_G^n(V, W)\) for all irreducible \(V\), \(W\) lying in \(r\)-blocks of cyclic defect in the simple groups \(\Sz(q)\), \(\PSU_3(q)\) and \(\Ree(q)\) in cross characteristic, obtaining in particular the dimensions of all cohomology groups for such modules. Along the way, we also obtain an analogous result for any \(r\)-block of cyclic defect whose Brauer tree is either a star or line (open polygon).
\end{abstract}


\section{Introduction} \label{sec:intro}

Given a finite group \(G\) and a field \(k\) of characteristic \(r\) dividing \(\abs{G}\), the groups \(\Ext_G^n(V, W)\) for \(V\), \(W \in \Irr_k G\) and in particular the cohomology groups \(\H^n(G, V)\) give important structural information about \(G\) and its representation theory over \(k\). For example, for finite groups of Lie type in defining characteristic, Scott and Sprowl \cite{ScottSprowl} and L\"ubeck \cite[Theorem 4.7]{LubeckComputation} have computed examples of the first cohomology in small rank groups which are notable for being significantly larger in dimension than any examples known beforehand (the largest known example prior to these was of dimension three), and in fact these computations were instrumental in disproving Wall's Conjecture (see \cite{WallConjectureCounterexample}).

In general, relatively little is known about the actual dimensions of these cohomology and Ext groups, especially for \(n > 1\). There exist some generic bounds, for example Guralnick and Tiep \cite{GuralnickTiep2} showed that if \(G\) is a finite group and \(V\) an irreducible \(kG\)-module then \(\dim \H^1(G,V)\) is bounded by some constant dependent only on \(r\) and the \emph{sectional \(r\)-rank} of \(G\). More recently, in the preprint \cite{GuralnickTiep3}, they have further generalised this to show that, for \(W\) irreducible, \(\dim \Ext_G^n(V, W)\) is bounded by some constant dependent only on the sectional \(r\)-rank of \(G\), \(\dim V\) and \(n\). Similarly, in \cite[Theorem 1.2.1]{DefiningHnBound} we see that if \(G\) is a finite group of Lie type and \(r = p\) then \(\dim \Ext_G^n(V, W)\) is bounded by some constant dependent upon the root system of \(G\), \(n\) and some information on the weight corresponding to \(V\) with a bound on cohomology dependent only upon \(n\) and the root system. As for more explicit bounds or examples, in \cite{GuralnickPresentations} it is shown that \(\dim \H^2(G, V) \leq \frac{35}{2} \dim V\) for \(G\) quasisimple and \(V\) arbitrary, or \(\frac{37}{2} \dim V\) for \(G\) arbitrary and \(V\) faithful and irreducible.

It is clear that the Ext groups for finite simple groups are an important family of examples (see, for example \cite[Theorem 1.4]{GuralnickTiep2}, \cite[Theorem 1.4]{GuralnickTiep3}), and in particular if one wishes to know more about these Ext groups in finite groups of Lie type in cross characteristic, one common approach is to first consider the rank one groups and investigate the implications of this in groups of higher ranks. To this end, in a previous work \cite{Paper2}, we determined \(\Ext_G^n(V, W)\) for all \(n\) and all \(V\), \(W \in \Irr_k G\) for \(G \in \{\PSL_2(q), \PGL_2(q), \SL_2(q)\}\) where \(0 < r \nmid q\).

Furthering this investigation into the Ext groups of rank one groups of Lie type, we look into the cases where the Sylow \(r\)-subgroups of \(G\) are cyclic, yielding the following result.

\begin{introthm}	\label{Main}
	Let \(k\) be an algebraically closed field of characteristic \(r\) and suppose \(G \in \{\Sz(q), \Ree(q), \PSU_3(q)\}\) has cyclic Sylow \(r\)-subgroups with \((r, q) = 1\). Then, for all \(V\), \(W \in \Irr_k G\),  \(\dim \Ext_G^n(V, W)\) is as in \cref{sec:Unitary,Suzuki1,SuzukiCase2,SuzukiCase3,ReeMinusOne,forkylad,ReeStar,ReeLongStar}.
\end{introthm}

In the case of the above theorem, since we work in the case where the Sylow \(r\)-subgroups of \(G\) are cyclic, to each \(r\)-block of \(G\) we may associate a \emph{Brauer tree}: a graph (indeed, a tree) which completely encodes the structure of the projective indecomposable modules (PIMs) in this block. In \cite[Theorem 1.2]{FeitBrauerTrees}, it is shown that for an arbitrary \(r\)-block \(B\) with nontrivial cyclic defect group of a finite group \(G\), either the Brauer graph of \(B\) has at most 248 edges or is a \emph{line} (often called an \emph{open polygon}). It is also a consequence of \cite[Theorem 1.1]{FeitBrauerTrees} that if \(G\) is \(r\)-soluble then the Brauer tree of \(B\) is a \emph{star}. Note here that by a \emph{star} we mean a complete bipartite graph \(K_{1, n-1}\) and by a \emph{line} we mean a path \(P_n\).

On the way to \cref{Main}, we also prove the following result which, by the above, yields the dimensions of all \(\Ext\)s between irreducible modules in any cyclic block of an \(r\)-soluble group and a large number of other cyclic blocks.

\begin{introthm}	\label{lines and stars}
	Let \(k\) be an algebraically closed field of characteristic \(r\) and \(B\) be a \(r\)-block of a finite group \(G\) whose Brauer tree is a star or a line. Then, for all irreducible \(kG\)-modules \(V\) and \(W\) lying in \(B\), \(\dim \Ext_G^n(V, W)\) is as in \cref{StarCohomologyExceptionalMiddle,StarCohomologyExceptionalOuter,LineCohomologyExceptionalOuter,LineCohomologyExceptionalInner}.
\end{introthm}

Our main tool in this paper is the \emph{Heller translate} \(\Omega V\) of a \(kG\)-module \(V\). Let \(\epsilon \colon \PC(V) \surj V\) be a surjective map onto \(V\) from its projective cover. Then \(\Omega V \coloneqq \ker \epsilon\) and \(\Omega^{n+1} V \coloneqq \Omega (\Omega^n V)\). By \cref{OmegaCohomology}, when \(W\) is irreducible we have that \(\dim \Ext_G^n(V, W) = \dim \Hom_G(\Omega^n V, W)\) and so the problem of determining \(\dim \Ext_G^n(V, W)\) may be reduced to a matter of determining the structure of the Heller translates of various modules.

To be able to calculate Heller translates one should probably know something about the projective covers of various irreducible \(kG\)-modules. As mentioned above, the structure of these PIMs is encoded by the Brauer tree of the corresponding \(r\)-block of \(G\). As such, we may use existing information on the Brauer trees of groups of Lie type \cite{HissReeBrauerTrees,GeckUnitaryRepresentations,BurkhardtSuzuki} to determine the structure of the PIMs required for our calculations.

\section{Preliminaries} \label{sec:prelim}

Our main tool throughout this article will be the following

\begin{defn}
	Let \(V\) be a \(kG\)-module with projective cover \(\PC(V)\). Let \(\epsilon \colon \PC(V) \surj V\) be a surjective map onto \(V\). Then we define \(\Omega V \coloneqq \ker \epsilon\) to be the \emph{Heller translate} of \(V\) and \(\Omega^{n+1} V \coloneqq \Omega (\Omega^n V)\).
\end{defn}

We also use the below lemma without reference.

\begin{lemmacite}[{\cite[Proposition 1]{HellerIndecomposable}}]
	The Heller translate \(\Omega\) is a permutation on the set of isomorphism classes of non-projective indecomposable \(kG\)-modules.
\end{lemmacite}

The below lemma then links the Heller translate to cohomology and \(\Ext\).

\begin{lemmacite}[{\cite[Lemma 1]{AlperinOmega}}] \label{OmegaCohomology}
	Let \(U\), \(V\) be \(kG\)-modules with \(V\) irreducible. Then, for any \(n \geq 0\),
	\[\Ext_G^n(U,V) \cong \Hom_G(\Omega^n U, V).\]
\end{lemmacite}

From the above it is clear that if we intend to use the Heller translate to determine the dimensions of \(\Ext\) groups, we will want to know more about the structure of the PIMs for \(G\).

\section{Generalities} \label{sec:General}

We take this result from \cite[Proposition 2.14]{Paper2}.

\begin{propn} \label{LonelyModule}
	Let \(B\) be a block containing a single non-projective simple module \(V\) with a cyclic defect group. Then \(\Ext_G^n(V,V) \cong k\) for all \(n\).
\end{propn}

\begin{defn}	\label{StarDef}
	A \emph{star} is a tree with \(n\) vertices, one of which has degree \(n-1\) and all others have degree 1 (alternatively, one may regard this as a complete bipartite graph \(K_{1,n-1}\)). If \(k\) is a field of characteristic \(p\), \(B\) a \(p\)-block of a finite group and the Brauer tree of \(kB\) is a star then we may represent this as on the left in \cref{StarsAndLines}. In such a case, the exceptional vertex (if it exists) is either the central vertex (as drawn) or any other vertex. We choose our notation so that if the exceptional vertex is an outer one, it is connected to the simple module \(S_1\).
\end{defn}

The following is immediate from the Brauer tree in question.

\begin{propn}	\label{StarProjectives}
	Suppose that \(B\) is a \(p\)-block of a finite group \(G\) whose Brauer tree is a star (this is simply the graph \(P_n\) for some \(n\)). If the exceptional vertex is the central vertex with exceptionality \(m \geq 1\) then the projective cover of \(S_i\) has shape \(\PC(S_i) \sim [S_i \mid M_i \mid S_i \mid \cdots \mid M_i \mid S_i]\), where \(M_i \sim [S_{i+1} \mid S_{i+2} \mid \cdots \mid S_{i+n-1}]\) with indices taken mod \(n\). Otherwise the exceptional vertex is an outer vertex with exceptionality \(m > 1\) and \(\PC(S_i) \sim [S_i \mid M_i \mid S_i]\) for all \(i \neq 1\) and \(\PC(S_1) \sim [S_1 \mid M_1 \oplus N \mid S_1]\) where \(N \sim [S_1 \mid S_1 \mid \cdots \mid S_1]\) is uniserial and contains \(S_1\) as a composition factor with multiplicity \(m\).
\end{propn}

\begin{defn}	\label{LineDef}
	A \emph{line} (also called an \emph{open polygon}) is a tree with two vertices of degree one and a Hamiltonian path between them (that is, a path including all vertices). If \(k\) is a field of characteristic \(p\), \(B\) a \(p\)-block of a finite group and the Brauer tree of \(kB\) is a line then we may represent this as on the right in \cref{StarsAndLines}. In such a case, the exceptional vertex (if it exists) is either an endpoint or an interior vertex. We choose our notation so that if the exceptional vertex is an outer one then it is connected to the simple module \(S_1\) and otherwise it is connected to the simple modules \(S_a\) and \(S_{a+1}\) with \(a \geq \frac{n}{2}\).
\end{defn}

Again the following is immediate from the Brauer tree.

\begin{propn}	\label{LineProjectives}
	Suppose that \(B\) is a \(p\)-block of a finite group \(G\) whose Brauer tree is a line. If the exceptional vertex is an outer vertex with exceptionality \(m > 1\) then \(\PC(S_i) \sim [S_i \mid S_{i-1} \oplus S_{i+1} \mid S_i]\) for all \(1 < i < n\), \(\PC(S_n) \sim [S_n \mid S_{n-1} \mid S_n]\) and \(\PC(S_1) \sim [S_1 \mid S_2 \oplus N \mid S_1]\). Otherwise the exceptional vertex is inner with exceptionality \(m \geq 1\) and \(\PC(S_i) \sim [S_i \mid [S_{i-1} \oplus S_{i+1} \mid S_i]\) for all \(1 < i < n\) not \(a\) or \(a+1\), \(\PC(S_1) \sim [S_1 \mid S_2 \mid S_1]\), \(\PC(S_n) \sim [S_n \mid S_{n-1} \mid S_n]\). Finally, \(\PC(S_a) \sim [S_a \mid S_{a-1} \oplus B_a \mid S_a]\) and \(\PC(S_{a+1}) \sim [S_{a+1} \mid B_{a+1} \oplus S_{a+2} \mid S_{a+1}]\) where \(B_a \sim [S_{a+1} \mid S_a \mid S_{a+1} \mid \cdots \mid S_{a+1}]\) contains each of \(S_a\) and \(S_{a+1}\) as composition factors with multiplicity \(m\) and \(B_{a+1}\) is similar but with \(a\) and \(a+1\) swapped.
\end{propn}

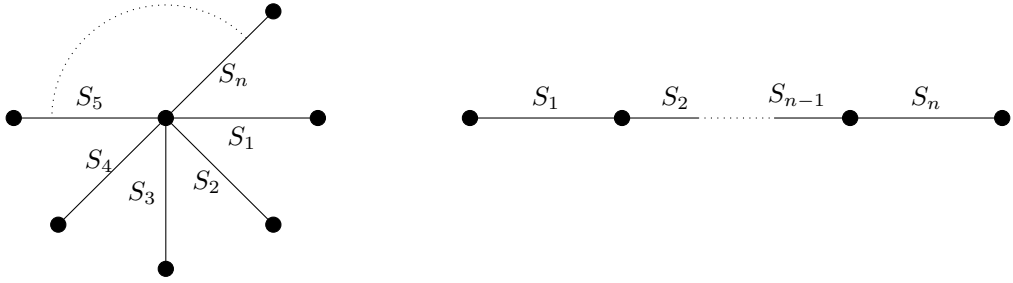
\begin{figure}[h]
	\begin{tikzpicture}
		\coordinate (O) at (-4,0);	
		\filldraw (O) circle(0.1);
		\filldraw (O) ++(0:2) circle(0.1);
		\filldraw (O) ++(-45:2) circle(0.1);
		\filldraw (O) ++(-90:2) circle(0.1);
		\filldraw (O) ++(-135:2) circle(0.1);
		\filldraw (O) ++(-180:2) circle(0.1);
		\filldraw (O) ++(45:2) circle(0.1);

		\draw (O) -- ++(0:2) node[pos=0.5, below=0pt] {\(S_1\)};
		\draw (O) -- ++(-45:2) node[pos=0.6, below=0pt, left=0pt] {\(S_2\)};
		\draw (O) -- ++(-90:2) node[pos=0.5, left=0pt] {\(S_3\)};
		\draw (O) -- ++(-135:2) node[pos=0.4, above=0pt, left=0pt] {\(S_4\)};
		\draw (O) -- ++(-180:2) node[pos=0.5, above=0pt] {\(S_5\)};
		\draw (O) -- ++(45:2) node[pos=0.4, below=0pt, right=0pt] {\(S_n\)};

		\draw[dotted, domain=45:180] plot ({1.5*cos(\x)-4}, {1.5*sin(\x)});

		\filldraw (0,0) circle(0.1);
		\filldraw (2,0) circle(0.1);
		\filldraw (5,0) circle(0.1);
		\filldraw (7,0) circle(0.1);

		\draw (0, 0) -- (2, 0) node[pos=0.5, above=0pt] {\(S_1\)};
		\draw (2, 0) -- (3, 0) node[pos=0.7, above=0pt] {\(S_2\)};
		\draw[dotted] (3, 0) -- (4, 0);
		\draw (4, 0) -- (5, 0) node[pos=0.3, above=0pt] {\(S_{n-1}\)};
		\draw (5, 0) -- (7, 0) node[pos=0.5, above=0pt] {\(S_n\)};
	\end{tikzpicture}
	\caption{A star and a line.}	\label{StarsAndLines}
\end{figure}

\begin{propn} \label{StarCohomologyExceptionalMiddle}
	Suppose \(G\) is a finite group and \(B\) is a block of \(kG\) whose Brauer tree is a star with exceptional vertex of exceptionality \(m \geq 1\) at its centre. Suppose that the irreducible \(kG\)-modules lying in \(B\) are \(\{S_1, \ldots, S_n\}\). Then \(\Ext_G^l(S_i, S_j)\) is 1-dimensional when \(l \equiv j-i\) or \(j-i + 1 \mod 2n\) and zero otherwise.
\end{propn}

\begin{proof}
	Recall the structure of the PIMs in this case from \cref{StarProjectives}. It is then easy to check that \(\Omega S_i \sim [S_{i+1} \mid M_{i+1} \mid \ldots \mid M_{i+1} \mid S_i] \sim \rad \PC(S_{i+1})\), so \(\Omega^2 S_i \cong S_{i+1}\) and the result follows from \cref{OmegaCohomology}.
\end{proof}

\begin{propn} \label{StarCohomologyExceptionalOuter}
	Suppose \(G\) is a finite group and \(B\) is a block of \(kG\) whose Brauer tree is a star with exceptional vertex of exceptionality \(m > 1\) an outer vertex. Suppose that the irreducible \(kG\)-modules lying in \(B\) are \(\{S_1, \ldots, S_n\}\) with \(S_1\) connected to the exceptional vertex. Then \(\Ext_G^l(S_i, S_j)\) is 1-dimensional when \(i = j = 1\) or \(l \equiv j - i\) or \(j - i - 1 \mod 2n\) and zero otherwise.
\end{propn}

\begin{proof}
	Recall again the structure of the PIMs from \cref{StarProjectives}. If \(i\) is neither 1 nor \(n\) then \(\Omega^2 S_i \cong S_{i+1}\) as in \cref{StarCohomologyExceptionalMiddle}. It is then easily calculated that \(\Omega S_n \sim [S_1 \mid M_1]\), \(\Omega^2 S_n \sim [N \mid S_1]\) and \(\Omega^3 S_n \sim [M_1 \mid S_1]\) so that \(\Omega^4 S_n \sim S_2\). 

	This gives us the dimensions of all but \(\Ext_G^l(S_1, S_1)\). We claim that \(\Ext_G^l(S_1, S_1)\) is 1-dimensional for all \(l\). Suppose not, and in particular suppose that \(a > 0\) is minimal such that \(\Ext_G^a(S_1, S_1) = 0\). Then we have that \(\Omega^a S_1\) must be a quotient of \(\PC(S_i)\) for some \(i \neq 1\) as \(\PC(S_1)\) is the only PIM which is not uniserial. In particular, this means that \(\Omega^a S_1 \sim [S_i \mid S_{i+1} \mid \ldots \mid S_{i+j}]\) for \(j < n\) and indices taken mod \(n\). Then we see that \(\Omega^{a-1} S_1 \sim [S_{i+j} \mid S_{i+j + 1} \mid \ldots \mid S_{i-1}] \sim [S_1 \mid S_2 \mid \ldots \mid S_{i-1}]\) by the minimality of \(a\). But then \(\Omega^{a-2} S_1 \sim [S_{i-1} \mid S_i \mid \ldots \mid S_n]\) and so \(\Ext_G^{a-2}(S_1, S_1) = 0\), contradicting the minimality of \(a\).
\end{proof}

To deal with the case where the Brauer tree is a line, we set up some notation which we adapt from \cite[2]{DudasLineExtAlgebra}. We leave the definitions of \(\twist{i}X^j\), \(\twist{i}Y_j\), \({}_i X^j\) and \({}_i Y_j\) unchanged for \(i\), \(j \geq 1\) but we also define \(\twist{-i}X^j\), \({}^{-i}Y^j\), \(\twist{i}Z^j\) and similar variants thereof. For positive integers \(i\), \(j\), we define \(\twist{-i}X^j\) to be the unique indecomposable module of shape
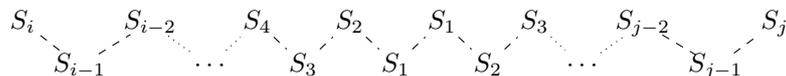
\begin{center}
	\begin{tikzpicture}
		\matrix(A)[matrix of math nodes, nodes in empty cells] {
			S_i	&			&	S_{i-2}	&			&	S_4	&		&	S_2	&		&	S_1	&		&	S_3	&			&	S_{j-2}	&			&	S_j	\\
				&	S_{i-1}	&			&	\cdots	&		&	S_3	&		&	S_1	&		&	S_2	&		&	\cdots	&			&	S_{j-1}	&		\\
		};

		\draw[dashed, shorten <>= 0.3cm] (A-1-1.center) -- (A-2-2.center);
		\draw[dashed, shorten <>= 0.3cm] (A-2-2.center) -- (A-1-3.center);
		\draw[dotted, shorten <>= 0.3cm] (A-1-3.center) -- (A-2-4.center);
		\draw[dotted, shorten <>= 0.3cm] (A-2-4.center) -- (A-1-5.center);
		\draw[dashed, shorten <>= 0.3cm] (A-1-5.center) -- (A-2-6.center);
		\draw[dashed, shorten <>= 0.3cm] (A-2-6.center) -- (A-1-7.center);
		\draw[dashed, shorten <>= 0.3cm] (A-1-7.center) -- (A-2-8.center);
		\draw[dashed, shorten <>= 0.3cm] (A-2-8.center) -- (A-1-9.center);
		\draw[dashed, shorten <>= 0.3cm] (A-1-9.center) -- (A-2-10.center);
		\draw[dashed, shorten <>= 0.3cm] (A-2-10.center) -- (A-1-11.center);
		\draw[dotted, shorten <>= 0.3cm] (A-1-11.center) -- (A-2-12.center);
		\draw[dotted, shorten <>= 0.3cm] (A-2-12.center) -- (A-1-13.center);
		\draw[dashed, shorten <>= 0.3cm] (A-1-13.center) -- (A-2-14.center);
		\draw[dashed, shorten <>= 0.3cm] (A-2-14.center) -- (A-1-15.center);
	\end{tikzpicture}
\end{center}
This means that the modules in the bottom `row' above are all submodules of \(\twist{-i}X^j\) and those in the top `row' all lie in the quotient by the bottom `row' with the dashed lines representing non-split extensions. For example, taking a quotient by the unique submodule isomorphic to \(S_{i-1}\) would yield a module \(S_i \oplus \twist{i-2}X^j\) (and \(S_i\) is not a submodule of \(\twist{-i}X^j\) provided \(i > j\)). Similarly, we define \({}^{-i}Y^j\) to be the unique indecomposable module of shape
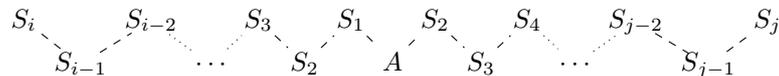
\begin{center}
	\begin{tikzpicture}
		\matrix(A)[matrix of math nodes, nodes in empty cells] {
			S_i	&			&	S_{i-2}	&			&	S_3	&		&	S_1	&		&	S_2	&		&	S_4	&			&	S_{j-2}	&			&	S_j	\\
				&	S_{i-1}	&			&	\cdots	&		&	S_2	&		&	A	&		&	S_3	&		&	\cdots	&			&	S_{j-1}	&		\\
		};

		\draw[dashed, shorten <>= 0.3cm] (A-1-1.center) -- (A-2-2.center);
		\draw[dashed, shorten <>= 0.3cm] (A-2-2.center) -- (A-1-3.center);
		\draw[dotted, shorten <>= 0.3cm] (A-1-3.center) -- (A-2-4.center);
		\draw[dotted, shorten <>= 0.3cm] (A-2-4.center) -- (A-1-5.center);
		\draw[dashed, shorten <>= 0.3cm] (A-1-5.center) -- (A-2-6.center);
		\draw[dashed, shorten <>= 0.3cm] (A-2-6.center) -- (A-1-7.center);
		\draw[dashed, shorten <>= 0.3cm] (A-1-7.center) -- (A-2-8.center);
		\draw[dashed, shorten <>= 0.3cm] (A-2-8.center) -- (A-1-9.center);
		\draw[dashed, shorten <>= 0.3cm] (A-1-9.center) -- (A-2-10.center);
		\draw[dashed, shorten <>= 0.3cm] (A-2-10.center) -- (A-1-11.center);
		\draw[dotted, shorten <>= 0.3cm] (A-1-11.center) -- (A-2-12.center);
		\draw[dotted, shorten <>= 0.3cm] (A-2-12.center) -- (A-1-13.center);
		\draw[dashed, shorten <>= 0.3cm] (A-1-13.center) -- (A-2-14.center);
		\draw[dashed, shorten <>= 0.3cm] (A-2-14.center) -- (A-1-15.center);
	\end{tikzpicture}
\end{center}
We also define \(\twist{-i}X^{-j} \coloneqq \twist{j}X^i \eqqcolon {}^{-i}Y^{-j}\). Finally, fix some \(1 < a < n\) and let \(B_a\) be a uniserial module with composition factors \(S_{a+1}\), \(S_a\), \(S_{a+1}\), \ldots, \(S_{a+1}\) and \(B_{a+1}\) be uniserial with composition factors \(S_a\), \(S_{a+1}\), \(S_a\), \ldots, \(S_a\). Then we define \(\twist{i}Z^j\) for \(i \equiv a \mod 2\) to be the unique indecomposable module of shape
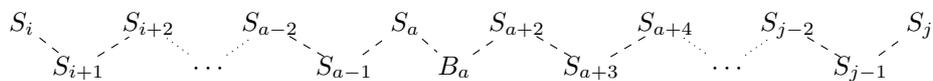
\begin{center}
	\begin{tikzpicture}
		\matrix(A)[matrix of math nodes, nodes in empty cells] {
			S_i	&			&	S_{i+2}	&			&	S_{a-2}	&			&	S_a	&		&	S_{a+2}	&			&	S_{a+4}	&			&	S_{j-2}	&			&	S_j	\\
				&	S_{i+1}	&			&	\cdots	&			&	S_{a-1}	&		&	B_a	&			&	S_{a+3}	&			&	\cdots	&			&	S_{j-1}	&		\\
		};

		\draw[dashed, shorten <>= 0.3cm] (A-1-1.center) -- (A-2-2.center);
		\draw[dashed, shorten <>= 0.3cm] (A-2-2.center) -- (A-1-3.center);
		\draw[dotted, shorten <>= 0.3cm] (A-1-3.center) -- (A-2-4.center);
		\draw[dotted, shorten <>= 0.3cm] (A-2-4.center) -- (A-1-5.center);
		\draw[dashed, shorten <>= 0.3cm] (A-1-5.center) -- (A-2-6.center);
		\draw[dashed, shorten <>= 0.3cm] (A-2-6.center) -- (A-1-7.center);
		\draw[dashed, shorten <>= 0.3cm] (A-1-7.center) -- (A-2-8.center);
		\draw[dashed, shorten <>= 0.3cm] (A-2-8.center) -- (A-1-9.center);
		\draw[dashed, shorten <>= 0.3cm] (A-1-9.center) -- (A-2-10.center);
		\draw[dashed, shorten <>= 0.3cm] (A-2-10.center) -- (A-1-11.center);
		\draw[dotted, shorten <>= 0.3cm] (A-1-11.center) -- (A-2-12.center);
		\draw[dotted, shorten <>= 0.3cm] (A-2-12.center) -- (A-1-13.center);
		\draw[dashed, shorten <>= 0.3cm] (A-1-13.center) -- (A-2-14.center);
		\draw[dashed, shorten <>= 0.3cm] (A-2-14.center) -- (A-1-15.center);
	\end{tikzpicture}
\end{center}
and for \(i \equiv a + 1 \mod 2\) it is instead defined to be the unique indecomposable module of the below shape.
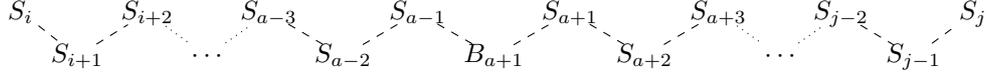
\begin{center}
	\begin{tikzpicture}
		\matrix(A)[matrix of math nodes, nodes in empty cells] {
			S_i	&			&	S_{i+2}	&			&	S_{a-3}	&			&	S_{a-1}	&			&	S_{a+1}	&			&	S_{a+3}	&			&	S_{j-2}	&			&	S_j	\\
				&	S_{i+1}	&			&	\cdots	&			&	S_{a-2}	&			&	B_{a+1}	&			&	S_{a+2}	&			&	\cdots	&			&	S_{j-1}	&		\\
		};

		\draw[dashed, shorten <>= 0.3cm] (A-1-1.center) -- (A-2-2.center);
		\draw[dashed, shorten <>= 0.3cm] (A-2-2.center) -- (A-1-3.center);
		\draw[dotted, shorten <>= 0.3cm] (A-1-3.center) -- (A-2-4.center);
		\draw[dotted, shorten <>= 0.3cm] (A-2-4.center) -- (A-1-5.center);
		\draw[dashed, shorten <>= 0.3cm] (A-1-5.center) -- (A-2-6.center);
		\draw[dashed, shorten <>= 0.3cm] (A-2-6.center) -- (A-1-7.center);
		\draw[dashed, shorten <>= 0.3cm] (A-1-7.center) -- (A-2-8.center);
		\draw[dashed, shorten <>= 0.3cm] (A-2-8.center) -- (A-1-9.center);
		\draw[dashed, shorten <>= 0.3cm] (A-1-9.center) -- (A-2-10.center);
		\draw[dashed, shorten <>= 0.3cm] (A-2-10.center) -- (A-1-11.center);
		\draw[dotted, shorten <>= 0.3cm] (A-1-11.center) -- (A-2-12.center);
		\draw[dotted, shorten <>= 0.3cm] (A-2-12.center) -- (A-1-13.center);
		\draw[dashed, shorten <>= 0.3cm] (A-1-13.center) -- (A-2-14.center);
		\draw[dashed, shorten <>= 0.3cm] (A-2-14.center) -- (A-1-15.center);
	\end{tikzpicture}
\end{center}
The other variants \(\twist{-i}Y_j\), \({}_{-i}X^j\), \({}_{-i}Y_j\), etc. are all defined identically except the endpoints change exactly as in \cite{DudasLineExtAlgebra}.

We will first consider the case where the Brauer tree of \(B\) is a line with the exceptional vertex an outer vertex. We may suppose that the Brauer tree and simple modules \(S_i\) are as below.

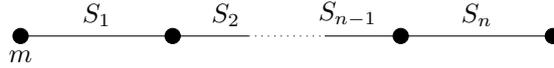
\begin{center}
	\begin{tikzpicture}
		\filldraw (-3.5,0) circle(0.1) node[below=2pt] {\(m\)};
		\filldraw (-1.5,0) circle(0.1);
		\filldraw (1.5,0) circle(0.1);
		\filldraw (3.5,0) circle(0.1);

		\draw (-3.5, 0) -- (-1.5, 0) node[pos=0.5, above=0pt] {\(S_1\)};
		\draw (-1.5, 0) -- (-0.5, 0) node[pos=0.7, above=0pt] {\(S_2\)};
		\draw[dotted] (-0.5, 0) -- (0.5, 0);
		\draw (0.5, 0) -- (1.5, 0) node[pos=0.3, above=0pt] {\(S_{n-1}\)};
		\draw (1.5, 0) -- (3.5, 0) node[pos=0.5, above=0pt] {\(S_n\)};
	\end{tikzpicture}
\end{center}

The dimensions of the \(\Ext\)s in this case will follow from computing the Heller translate of various \(X\), \(Y\) and \(Z\) above. To reduce repetition, note that we may look only at the endpoints (as drawn) of these modules and their types (\(X\), \(Y\) or \(Z\)). As such, we will state results for one endpoint and type at a time where possible. For example, \(\Omega(\twist{i}X) \cong \twist{i-1}X\) means both \(\Omega(\twist{i}X^j) \cong \twist{i-1}X^{j_1}\) and \(\Omega(\twist{i}Y_j) \cong \twist{i-1}Y_{j_2}\) for \(j_1\), \(j_2\) as given by the corresponding rules for \(X^j\) and \(Y_j\). We now begin by stating some rules which cover most of the simpler cases.
\newcommand{\llOmega}{{}_* \Omega}
\newcommand{\ulOmega}{{}^* \Omega}
\newcommand{\lrOmega}{\Omega_*}
\newcommand{\urOmega}{\Omega^*}

\begin{lemma}	\label{MaybeCoversMostModules}
	Let \(W \in \{X, Y, Z\}\). Define functions \(\llOmega\), \(\ulOmega\), \(\lrOmega\), \(\urOmega\) and \(\Omega'_{ij}\) (for \(i \neq -n\), \(j \neq n\) and, if there is no exceptional vertex, \(i \neq 1\)) such that 
	\[\Omega ({}_{i_l}^{i_u} W^{j_u}_{j_l}) \cong {}_{\llOmega(i_l)}^{\ulOmega(i_u)} (\Omega'_{ij} W)_{\lrOmega(j_l)}^{\urOmega(j_u)}\]
	where \(i\) is the value of whichever of \(i_u\) or \(i_l\) is defined, and similarly for \(j\). 
	Then we have the following:
	\begin{enumerate}[i)]
		\item if \(i < 0\), \(\Omega_{ij}' X = Y\);
		\item \(\Omega_{ij}' Y = X\);
		\item if the exceptional vertex is an inner vertex and \(i < a\) or \(j > a+1\) we have \(\Omega_{ij}' X = Z\), otherwise \(\Omega_{ij}'X = X\);
		\item \(\Omega_{ij}' Z = X\);
		\item if \(i \neq \pm 1\), \(-n\) and \(j \neq \pm 1\), \(n\) we have \( \llOmega (i) = i+1\), \(\ulOmega(i) = i-1\), \(\lrOmega (j) = j-1\) and \(\urOmega(j) = j+1\),
		\item in all cases, \(\Omega W^n \cong (\Omega_{in}'(W))_n\) and \(\Omega (\twist{-n}W) \cong {}_{-n}(\Omega_{-nj}'(W))\).
	\end{enumerate}
\end{lemma}

\begin{proof}
	For this, we only consider the section of the modules in question made up of the composition factors which are `close to' the exceptional vertex in the Brauer graph as the rest is identical to the arguments used in \cite[Lemma 2.1]{DudasLineExtAlgebra}. For i), note that for suitable \(i\) and \(j\), \(X\) has a section \({}_{-3} Y_2\) of shape

	\begin{center}
		\begin{tikzpicture}
			\matrix(A)[matrix of math nodes, nodes in empty cells] {
						&	S_2	&			&	S_1		&			\\
				S_3		&		&	S_1		&			&	S_2		\\
			};

			\draw[dashed, shorten <>= 0.3cm] (A-2-1.center) -- (A-1-2.center);
			\draw[dashed, shorten <>= 0.3cm] (A-1-2.center) -- (A-2-3.center);
			\draw[dashed, shorten <>= 0.3cm] (A-2-3.center) -- (A-1-4.center);
			\draw[dashed, shorten <>= 0.3cm] (A-1-4.center) -- (A-2-5.center);
		\end{tikzpicture}
	\end{center}
	with projective cover
	\[\threemoduleshape{S_2}{S_3 \oplus S_1}{S_2} \oplus \threemoduleshape{S_1}{A \oplus S_2}{S_1}.\]
	It is then easy to see here that the kernel of such a cover must have shape \([S_1 \mid S_2 \oplus A]\) as in \(Y\). When the above section of \(X\) is not the whole module, one may patch this together as in \cite[Lemma 3.18]{Paper2} to obtain i), and when \(i\) or \(j\) are too small only minor modifications are needed. Parts ii)--iv) follow from an identical argument examining the sections
	\begin{center}
		\begin{tikzpicture}
			\matrix(A)[matrix of math nodes, nodes in empty cells] {
					&	S_1	&		&	S_2	&		&	&			&	S_a	&			&	S_{a+2}	&			&	&			&	S_{a-1}	&		&	S_{a+1}	&			\\	
				S_2	&		&	A	&		&	S_3	&	&	S_{a-1}	&		&	S_{a+1}	&			&	S_{a+3}	&	&	S_{a-2}	&			&	S_a	&			&	S_{a+2}	\\
			};

			\draw[dashed, shorten <>= 0.3cm] (A-2-1.center) -- (A-1-2.center);
			\draw[dashed, shorten <>= 0.3cm] (A-1-2.center) -- (A-2-3.center);
			\draw[dashed, shorten <>= 0.3cm] (A-2-3.center) -- (A-1-4.center);
			\draw[dashed, shorten <>= 0.3cm] (A-1-4.center) -- (A-2-5.center);

			\draw[dashed, shorten <>= 0.3cm] (A-2-7.center) -- (A-1-8.center);
			\draw[dashed, shorten <>= 0.3cm] (A-1-8.center) -- (A-2-9.center);
			\draw[dashed, shorten <>= 0.3cm] (A-2-9.center) -- (A-1-10.center);
			\draw[dashed, shorten <>= 0.3cm] (A-1-10.center) -- (A-2-11.center);

			\draw[dashed, shorten <>= 0.3cm] (A-2-13.center) -- (A-1-14.center);
			\draw[dashed, shorten <>= 0.3cm] (A-1-14.center) -- (A-2-15.center);
			\draw[dashed, shorten <>= 0.3cm] (A-2-15.center) -- (A-1-16.center);
			\draw[dashed, shorten <>= 0.3cm] (A-1-16.center) -- (A-2-17.center);
		\end{tikzpicture}

		\begin{tikzpicture}
			\matrix(A)[matrix of math nodes, nodes in empty cells] {
					&	S_{a-1}	&			&	S_{a+1}	&			&	&			&	S_a	&		&	S_{a+2}	&			\\
			S_{a-2}	&			&	B_{a+1}	&			&	S_{a+2}	&	&	S_{a-1}	&		&	B_a	&			&	S_{a+3}	\\
			};

			\draw[dashed, shorten <>= 0.3cm] (A-2-1.center) -- (A-1-2.center);
			\draw[dashed, shorten <>= 0.3cm] (A-1-2.center) -- (A-2-3.center);
			\draw[dashed, shorten <>= 0.3cm] (A-2-3.center) -- (A-1-4.center);
			\draw[dashed, shorten <>= 0.3cm] (A-1-4.center) -- (A-2-5.center);

			\draw[dashed, shorten <>= 0.3cm] (A-2-7.center) -- (A-1-8.center);
			\draw[dashed, shorten <>= 0.3cm] (A-1-8.center) -- (A-2-9.center);
			\draw[dashed, shorten <>= 0.3cm] (A-2-9.center) -- (A-1-10.center);
			\draw[dashed, shorten <>= 0.3cm] (A-1-10.center) -- (A-2-11.center);
		\end{tikzpicture}
	\end{center}
	of \(Y\), \(X\), \(X\), \(Z\) and \(Z\), respectively (note the duplicate \(X\)s and \(Z\)s because this case depends on whether \(i \equiv a \mod 2\)). Finally, v) and vi) are obvious due to the structure of the projective covers of simple modules not incident to the exceptional vertex.
\end{proof}

\begin{lemma}	\label{LineCohomologyOuterOmegas}
	Suppose that \(B\) is a \(p\)-block of a finite group \(G\) whose Brauer tree is a line with exceptional vertex an outer vertex (without loss of generality, \(S_1\) is incident to the exceptional vertex). Then we have the following isomorphisms.
	\begin{multicols}{2}
		\begin{enumerate}[i)]
			\item \(\Omega (\twist{1}X) \cong {}^{-1}Y\)
			\item \(\Omega ({}^{-1}Y) \cong {}^{-2}X\)
			\item \(\Omega ({}_{-1}X) \cong {}_{1}Y\)
			\item \(\Omega Y_{1} \cong Y_{-1} \cong {}_1 X\)
		\end{enumerate}
	\end{multicols}
\end{lemma}

\begin{proof}
	As with the previous cases, we focus only on a small section of the module in question. To see i) and ii), note that the Heller translate of a section of shape \([S_1 \oplus S_3 \mid S_2]\) of \(\twist{1} X\) has shape \([S_1 \oplus S_2 \mid A \oplus S_3]\), yielding i), and the Heller translate of the section just obtained has shape
	\begin{center}
		\begin{tikzpicture}
			\matrix(A)[matrix of math nodes, nodes in empty cells] {
			S_2	&		&	S_1	&		&	S_3 \\
				&	S_1	&		&	S_2	&		\\
			};

			\draw[dashed, shorten <>= 0.3cm] (A-1-1.center) -- (A-2-2.center);
			\draw[dashed, shorten <>= 0.3cm] (A-2-2.center) -- (A-1-3.center);
			\draw[dashed, shorten <>= 0.3cm] (A-1-3.center) -- (A-2-4.center);
			\draw[dashed, shorten <>= 0.3cm] (A-2-4.center) -- (A-1-5.center);
		\end{tikzpicture}
	\end{center}
	For iii), the Heller translate of a section of shape \([S_1 \mid S_1 \oplus S_2]\) of \({}_{-1} X\) has shape \([S_2 \oplus S_4 \mid A \oplus S_3]\) as in \({}_1 Y\) and for iv) note that the Heller translate of \([S_1 \mid S_2 \oplus A]\) is \([S_2 \mid S_1]\).
\end{proof}

\begin{propn} \label{LineCohomologyExceptionalOuter}
	Suppose \(G\) is a finite group and \(B\) is a block of \(kG\) whose Brauer tree is a line with exceptional vertex of exceptionality \(m > 1\) an outer vertex as described in \cref{LineDef}. Then \(\dim \Ext_G^{\ell}(S_i, S_j)\) is nonzero (thus 1-dimensional) for \(j \neq 1\) precisely when either of the following hold (where values of \(\ell\) are taken modulo \(2n\))
	\[
		\begin{cases}
			\ell \equiv \abs{i - j} \mod 2	 			&	\text{if } \abs{i-j} \leq \ell \leq 2n - i - j 			\\
			\ell \equiv \abs{i - j} + 1 \mod 2			&	\text{if } i + j - 1 \leq \ell \leq 2n - \abs{i-j} - 1	\\
		\end{cases}
	\]
	Further, \(\Ext_G^{\ell} (S_i, S_1)\) is nonzero precisely when \(i-1 \leq \ell \leq 2n-i\).
\end{propn}

\begin{proof}
	This follows from the fact that we may regard \(S_i\) as \(\twist{i}X^i\) and then repeatedly use \cref{MaybeCoversMostModules,LineCohomologyOuterOmegas}. We illustrate below the case where \(i < \frac{n}{2}\), \(j > i\) and \(i \equiv j \mod 2\).

	For \(\ell < j-i\), \(S_j\) clearly does not lie in the head of \(\Omega^{\ell} S_i = \Omega^{\ell} \twist{i}X^i\) but \(\Omega^{j-i} S_i\) has \(\twist{i} X^j\) as a quotient. Then \(S_j\) will remain the head of every second Heller translate of this module until we reach \(\Omega^{2n - i - j} S_i\) since \(\Omega^{n-i} S_i\) has quotient \(\twist{i} X^n\) and \(\Omega^{n-j+1} (X^n) = Y_j\) which does not have \(S_j\) in its head; this gives nonzero \(\Ext\)s for \(j-i \leq \ell \leq 2n - i - j\) for \(\ell\) even. For the other range, note that \(\Omega^i S_i \cong \twist{-1}Y^{2i+1}\) and so \(\Omega^{i-1+j} S_i\) has quotient \(\twist{-j}Y^1\), giving \(\Ext_G^{i+j-1}(S_i, S_j) \neq 0\) but no earlier odd-indexed nonzero \(\Ext\). Again, we see that \(S_j\) remains in the head of every second Heller translate of this module until we reach \(\Omega^{2n - \abs{i-j} - 1}\) as \(\Omega^{n + i - 1}\) has quotient \(\twist{-n}X^{-i} \cong \twist{i}X^n\) and as before \(\Omega^{n-j+1} (X^n)\) does not have \(S_j\) in its head, so nor does \(\Omega^{2n + i - j - 1} S_i\), yielding the second range of values of \(\ell\) in the statement.

	Note that if \(j < i\) then \(\Omega^{2n-i+1} S_i\) has quotient \({}_{-i} Y_{-1} \cong {}_1 Y_i\) and \(\Omega^{j-1} ({}_1 X) \cong {}_j X\) which does not have \(S_j\) in its head, giving instead a range of \(i + j - 1 \leq \ell \leq 2n + j - i - 1 = 2n + \abs{i-j} - 1\). All remaining cases require only minor modifications to the above.
\end{proof}

We now consider the other case, where the Brauer tree of \(B\) is a line but the exceptional vertex is an inner vertex. Suppose that the exceptional vertex connects the simple modules \(S_a\) and \(S_{a+1}\), then we may assume the Brauer tree of \(B\) is as below.

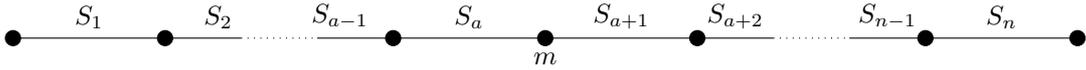
\begin{center}
	\begin{tikzpicture}
		\filldraw (-7,0) circle(0.1);
		\filldraw (-5,0) circle(0.1);
		\filldraw (-2,0) circle(0.1);
		\filldraw (0,0) circle(0.1) node[below=2pt] {\(m\)};
		\filldraw (2,0) circle(0.1);
		\filldraw (5,0) circle(0.1);
		\filldraw (7,0) circle(0.1);

		\draw (-7, 0) -- (-5, 0) node[pos=0.5, above=0pt] {\(S_1\)};
		\draw (-5, 0) -- (-4, 0) node[pos=0.7, above=0pt] {\(S_2\)};
		\draw[dotted] (-4, 0) -- (-3, 0);
		\draw (-3, 0) -- (-2, 0) node[pos=0.3, above=0pt] {\(S_{a-1}\)};
		\draw (-2, 0) -- (0, 0) node[pos=0.5, above=0pt] {\(S_a\)};
		\draw (0, 0) -- (2, 0) node[pos=0.5, above=0pt] {\(S_{a+1}\)};
		\draw (2, 0) -- (3, 0) node[pos=0.5, above=0pt] {\(S_{a+2}\)};
		\draw[dotted] (3, 0) -- (4, 0);
		\draw (4, 0) -- (5, 0) node[pos=0.5, above=0pt] {\(S_{n-1}\)};
		\draw (5, 0) -- (7, 0) node[pos=0.5, above=0pt] {\(S_n\)};
	\end{tikzpicture}
\end{center}

In this case, for \(i \neq 1\), \(a\), \(a+1\) or \(n\), we have that \(\PC(S_i) \sim [S_i \mid S_{i-1} \oplus S_{i+1} \mid S_i]\) as above and \(\PC(S_1) \sim [S_1 \mid S_2 \mid S_1]\), \(\PC(S_n) \sim [S_n \mid S_{n-1} \mid S_n]\), leaving us with \(\PC(S_a) \sim [S_a \mid S_{a-1} \oplus B_a \mid S_a]\) and \(\PC(S_{a+1}) \sim [S_{a+1} \mid B_{a+1} \oplus S_{a+2} \mid S_{a+1}]\) where \(B_a \sim [S_{a+1} \mid S_a \mid S_{a+1} \mid \ldots \mid S_{a+1}]\) and \(B_{a+1} \sim [S_a \mid S_{a+1} \mid S_a \mid \ldots \mid S_a]\).

\begin{lemma} \label{LineCohomologyInnerOmegas}
	Let \(B\) be a \(p\)-block of a finite group \(G\) whose Brauer tree is a line with exceptional vertex an inner vertex as described in \cref{LineDef}. We have \(\Omega(\twist{1}Z) \cong {}_1X\) and \(\Omega(\twist{1}X) \cong {}_1Z\) if \(S_a\) or \(S_{a+1}\) lie in the head of \(\twist{1}X\) and \({}_1 X\) otherwise. Further, \(\Omega {}_a Z_a \cong \twist{a-1} X^{a+1}\) and \(\Omega {}_{a+1} Z_{a+1} \cong \twist{a} X^{a+2}\).
\end{lemma}

\begin{proof}
	The first two statements are similar to \cref{MaybeCoversMostModules} iii), iv), thus left as an exercise, and the final statements follow easily from the fact that \({}_a Z_a \cong B_{a+1}\) and \({}_{a+1} Z_{a+1} \cong B_a\).
\end{proof}

\begin{propn} \label{LineCohomologyExceptionalInner}
	Suppose \(G\) is a finite group and \(B\) is a block of \(kG\) whose Brauer tree is a line with exceptional vertex of exceptionality \(m > 1\) an inner vertex as described in \cref{LineDef}. Then \(\dim \Ext_G^{\ell}(S_i, S_j)\) for \(i \neq a\) is nonzero (thus 1-dimensional) precisely when either of the following hold (where values of \(\ell\) are taken modulo \(2n\))
	\[
		\begin{cases}
			\ell \equiv \abs{i - j} \mod 2	 			&	\text{if } \abs{i-j} \leq \ell \leq \min\{i + j - 2, 2n - i - j\} 		\\
			\ell \equiv \abs{i - j} + 1 \mod 2			&	\text{if } \abs{n - i - j + 1} + n \leq \ell \leq 2n + \abs{i-j} - 1	\\	
		\end{cases}
	\]
	If \(a = \frac{n}{2}\) then \(\Ext_G^{\ell}(S_a, S_j)\) follows the same rules as above, otherwise \(a > \frac{n}{2}\) and we have \(\Ext_G^{\ell}(S_a, S_j)\) nonzero precisely when any of the below conditions hold.
	\[
		\begin{cases}
			\ell \equiv \abs{a - j} \mod 2				&	\text{if } \abs{a-j} 				\leq \ell \leq \min\{a + j - 2, 2n - a - j\}	\\
			a \geq \ell \equiv \abs{a - j} + 1 \mod 2	&	\text{if } 2n - a - j + 1 			\leq \ell \leq 2n + \abs{a-j} - 1				\\
			a < \ell \equiv \abs{a - j} + 1 \mod 2		&	\text{if } \abs{n - a - j + 1} + n 	\leq \ell \leq 2n + \abs{a-j} - 1				\\
		\end{cases}
	\]
\end{propn}

\begin{proof}
	As above, we may regard \(S_i\) as \(\twist{i} X^i\) and repeatedly use \cref{MaybeCoversMostModules,LineCohomologyInnerOmegas}. For \(i \neq a\) proceed in a manner identical to \cref{LineCohomologyExceptionalOuter} to obtain the result. We will illustrate here the case where \(j > i = a > \frac{n}{2}\) as this is slightly more complicated but still encompasses the arguments required for all other cases, with the case \(i = a+1\) being similar again so left as an exercise.

	Regard \(S_a\) as \(\twist{a} X^a\) and note that, as usual, \(S_j\) does not lie in the head of \(\Omega^{\ell} S_a\) for \(\ell < j-a\) but is in the head of \(\Omega^{j-a} S_a\) as well as every second Heller translate of this module until \(\Omega^{2n-a-j} S_a \cong W_j\) for \(W \in \{X, Z\}\) dependent on whether \(j \equiv a \mod 2\). Continue in this way until we reach \(\Omega^{2(n-a)} S_a \cong Y_{a+1}\) where we can no longer rely on \cref{MaybeCoversMostModules}. Take a section \({}_{a-3} Y_{a+1}\) of the aforementioned module and note that this has Heller translates
	
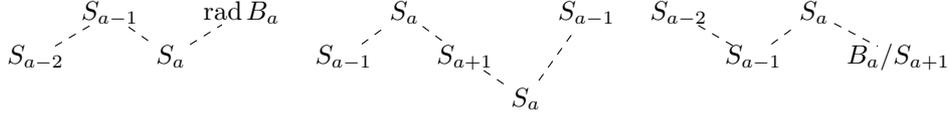
\begin{center}
		\begin{tikzpicture}
			\matrix(A)[matrix of math nodes, nodes in empty cells] {
					&	S_{a-1}	&		&	\rad B_a	&	&			&	S_a	&			&		&	S_{a-1}	&	&	S_{a-2}	&			&	S_a	&				\\
			S_{a-2}	&			&	S_a	&				&	&	S_{a-1}	&		&	S_{a+1}	&		&			&	&			&	S_{a-1}	&		&	B_a/S_{a+1}	\\
					&			&		&				&	&			&		&			&	S_a	&			&	&			&			&		&				\\
			};

			\draw[dashed, shorten <>= 0.3cm] (A-2-1.center) -- (A-1-2.center);
			\draw[dashed, shorten <>= 0.3cm] (A-1-2.center) -- (A-2-3.center);
			\draw[dashed, shorten <>= 0.3cm] (A-2-3.center) -- (A-1-4.center);

			\draw[dashed, shorten <>= 0.3cm] (A-2-6.center) -- (A-1-7.center);
			\draw[dashed, shorten <>= 0.3cm] (A-1-7.center) -- (A-2-8.center);
			\draw[dashed, shorten <>= 0.3cm] (A-2-8.center) -- (A-3-9.center);
			\draw[dashed, shorten <>= 0.3cm] (A-3-9.center) -- (A-1-10.center);

			\draw[dashed, shorten <>= 0.3cm] (A-1-12.center) -- (A-2-13.center);
			\draw[dashed, shorten <>= 0.3cm] (A-2-13.center) -- (A-1-14.center);
			\draw[dashed, shorten <>= 0.3cm] (A-1-14.center) -- (A-2-15.center);
		\end{tikzpicture}
	\end{center}
	with the third module above having Heller translate \(\twist{a-3}X^{a+1}\). Note also that if the second module above was attached to a module \([S_{a-3} \mid S_{a-2}]\) on the left as drawn above, the third Heller translate would instead be
	
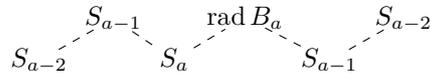
\begin{center}
		\begin{tikzpicture}
			\matrix(A)[matrix of math nodes, nodes in empty cells] {
					&	S_{a-1}	&		&	\rad B_a	&			&	S_{a-2}	\\
			S_{a-2}	&			&	S_a	&				&	S_{a-1}	&			\\
			};

			\draw[dashed, shorten <>= 0.3cm] (A-2-1.center) -- (A-1-2.center);
			\draw[dashed, shorten <>= 0.3cm] (A-1-2.center) -- (A-2-3.center);
			\draw[dashed, shorten <>= 0.3cm] (A-2-3.center) -- (A-1-4.center);
			\draw[dashed, shorten <>= 0.3cm] (A-1-4.center) -- (A-2-5.center);
			\draw[dashed, shorten <>= 0.3cm] (A-2-5.center) -- (A-1-6.center);
		\end{tikzpicture}
	\end{center}
	From this we can see that the Heller translates of \(\Omega^{2(n-a)} S_a\) will alternate between modules with sections of the second and fourth shapes above until reaching \(\Omega^{2a-1} S_a\) with section of the third shape (and no \(S_{a+1}\) in the socle), yielding \(\Omega^{2a} S_i \cong X^{a+1}\), whereupon the Heller translates resume alternating between \(X\) and \(Z\) until we reach \(\Omega^{2n} S_i \cong S_i\). Keeping track of the heads of these modules along the way yields the result.
\end{proof}

This leaves only the case where there is no exceptional vertex, which has already been done in \cite[Proposition 2.2]{DudasLineExtAlgebra} though may also be obtained by mimicking the proof of \cref{LineCohomologyExceptionalOuter} with \(\Omega \twist{1}X = {}_1 X\) rather than \({}^{-1} Y\).

\begin{propn}
	Suppose \(G\) is a finite group and \(B\) is a block of \(kG\) whose Brauer tree is a line with no exceptional vertex as described in \cref{LineDef}. Then \(\dim \Ext_G^{\ell}(S_i, S_j)\) is nonzero (thus 1-dimensional) precisely when either of the following hold (where values of \(\ell\) are taken modulo \(2n\))
	\[
		\begin{cases}
			\ell \equiv \abs{i - j} \mod 2	 			&	\text{if } \abs{i-j} \leq \ell \leq 2n - i - j 			\\
			\ell \equiv \abs{i - j} + 1 \mod 2			&	\text{if } i + j - 1 \leq \ell \leq 2n - \abs{i-j} - 1	\\
		\end{cases}
	\]
\end{propn}

\section{\texorpdfstring{\(\PSU_3(q)\)}{Unitary groups PSU(3,q)}} \label{sec:Unitary}

When the Sylow \(r\)-subgroups of \(G \coloneqq \PSU_3(q)\) are cyclic, the Brauer trees of all \(r\)-blocks of \(G\) are lines or a star with three points \cite{GeckUnitaryRepresentations}. In any such case, we are done by \cref{LineCohomologyExceptionalInner,StarCohomologyExceptionalMiddle,LonelyModule}. Otherwise, the Sylow \(r\)-subgroups of \(G\) are not cyclic, and when \(r > 2\) divides \(q+1\) or \(r = 2\) and \(q \equiv 3 \mod 4\) the representation type of \(kG\) is wild and the structure of the PIMs is not known. As such, determining the dimensions of cohomology or \(\Ext\) groups in this case is outside the scope of this work. 

\section{Extensions and cohomology in \texorpdfstring{\(\Sz(q)\)}{Suzuki groups}} \label{sec:Suzuki}

We next deal with the Suzuki groups \(G \coloneqq \Sz(q)\) (also denoted \(\SzB(q)\)) for \(q = 2^{2n+1}\) where \(n \geq 1\) (note here we mean that for \(n > 0\), \(G\) is a finite simple group and for \(a = 0\) we have \(G \cong C_5 \rtimes C_4\)), since the Sylow \(r\)-subgroups of these groups in non-defining characteristics are cyclic and so methods used before all apply. In particular, the Brauer trees are known due to Burkhardt \cite{BurkhardtSuzuki} for the Suzuki groups for all such cases and so we know the structure of the projective modules in these cases very well.

We will not require any structural information about \(\Sz(q)\) as the results are determined completely by the structure of the projective modules (thus the Brauer trees), but the curious reader should consult \cite[Chapters 13, 14]{CarterSimple} for more. These groups have order \(\abs{\Sz(q)} = q^2 (q-1) (q^2 + 1)\) which factors as \(q^2 (q-1) (q - s + 1) (q + s + 1)\), where \(s^2 = 2q\), and so the study of the cross characteristic representation theory of these groups splits naturally into the three cases \(r \mid q-1\), \(r \mid q - s + 1\) and \(r \mid q + s + 1\). For convenience, we reproduce the character table of \(\Sz(q)\) from \cite{BurkhardtSuzuki} and label the complex characters accordingly.

Let \(x\), \(y\) and \(z\) be elements of \(\Sz(q)\) of orders \(q-1\), \(q + s + 1\) and \(q - s + 1\), respectively, and let \(f\) and \(t\) be elements of respective orders 4 and 2. Powers of these elements give a set of conjugacy class representatives for \(G\). Now, let \(\omega\), \(\eta\) and \(\zeta\) be primitive \((q-1)\)\upth, \((q - s + 1)\)\upth{} and \((q+s+1)\)\upth{} complex roots of unity, respectively, and let \(\epsilon_d \coloneqq \zeta^d + \zeta^{-d} + \zeta^{qd} + \zeta^{-qd}\) and \(\delta_e \coloneqq \eta^e + \eta^{-e} + \eta^{qe} + \eta^{-qe}\). Finally, let \(a\), \(u \leq \frac{1}{2} (q-2)\); \(b\), \(l \leq \frac{1}{4}(q+s)\); \(c\), \(v \leq \frac{1}{4}(q-s)\) and of course \(i = \sqrt{-1}\), where \(a\), \(b\), \(c\), \(l\), \(u\) and \(v\) are all positive integers. Then the ordinary character table of \(G\) is as below.

\begin{table}[H]
	\[
		\begin{array}{lccccccc} \toprule
					&	1					&	x^a							&	y^b				&	z^c				&	t				&	f				&	f^{-1}				\\ \midrule
		1			&	1					&	1							&	1				&	1				&	1				&	1				&	1					\\
		\Pi			&	q^2					&	1							&	-1				&	-1				&	0				&	0				&	0					\\
		\Gamma_1	&	\frac{s}{2}(q-1)	&	0							&	1				&	-1				&	-\frac{s}{2}	&	\frac{si}{2}	&	-\frac{si}{2}		\\
		\Gamma_2	&	\frac{s}{2}(q-1)	&	0							&	1				&	-1				&	-\frac{s}{2}	&	-\frac{si}{2}	&	\frac{si}{2}		\\
		\Omega_u	&	q^2 + 1				&	\omega^{ua} + \omega^{-ua}	&	0				&	0				&	1				&	1				&	1					\\
		\Theta_l	&	(q-1)(q-s+1)		&	0							&	-\epsilon_{lb}	&	0				&	s-1				&	-1				&	-1					\\
		\Lambda_v	&	(q-1)(q+s+1)		&	0							&	0				&	-\delta_{vc}	&	-s-1			&	-1				&	-1					\\	\bottomrule
		\end{array}
	\]
	\caption{Character table for \(\Sz(q)\)}
\end{table}

We begin by considering the case where \(r \mid q-1\). In this case, the modules with characters \(\Gamma_i\), \(\theta_i\) and \(\Lambda_i\) lie in blocks of defect zero and thus are projective. The principal \(r\)-block of \(G\) consists of two modules, \(k\) and \(V\), where \(k\) is the trivial module as usual and \(\dim V = q^2\). The remaining modules lie alone in blocks of maximal defect. From \cite[423]{BurkhardtSuzuki}, the Brauer tree of the principal block is a line with two edges whose central vertex is exceptional with exceptionality \(m = \frac{r^x - 1}{2}\), where \(r^x\) is the \(r\)-part of \(q-1\). This is the same Brauer tree as in \cite[Proposition 3.6]{Paper2}, and so we already know the answer for the principal block. Alternatively, we may also use \cref{LineCohomologyExceptionalInner} to obtain the following result with the final statement following from \cref{LonelyModule}.

\begin{propn} \label{Suzuki1}
	Let \(G = \Sz(q)\) and suppose that \(r\) is an odd prime dividing \(q - 1\). Let \(V\) be the nontrivial irreducible module lying in the principal block. Then \(V = V^*\), thus \(\H^n(G,V) \cong \Ext_G^n(V,k)\), for all \(n\) and
	\[\H^n(G,k) \cong \Ext_G^n(V,V) \cong \begin{cases}
		0	&	n \equiv 1, \, 2 \mod 4,\\
		k	&	n \equiv 0, \, 3 \mod 4,
	\end{cases} \quad
	\H^n(G,V) \cong \begin{cases}
		k	&	n \equiv 1, \, 2 \mod 4,\\
		0	&	n \equiv 0, \, 3 \mod 4.
	\end{cases}\]
	Any non-projective irreducible module \(W\) lying outside the principal block (so \(W\) has character \(\Omega_i\) for some \(i\)) is such that \(\Ext_G^n(W, W) \cong k\) for all \(n\).
\end{propn}
%
%

We now consider the case where \(r \mid q - s + 1\). In this case, the modules with characters \(\Omega_i\) and \(\theta_i\) are projective and the principal \(r\)-block of \(G\) contains 4 modules \(k\), \(U\), \(V\) and \(W\). Here, \(V\) has dimension \(q^2 - 1\) and \(U \cong W^*\) each have dimension \(\frac{s}{2}(q-1)\). The remaining modules lie alone in blocks of maximal defect. From \cite[424]{BurkhardtSuzuki}, the Brauer tree for the principal \(r\)-block of \(G\) is as below with exceptionality \(m = \frac{r^x - 1}{4}\) where \(r^x\) is the \(r\)-part of \(q - s + 1\).

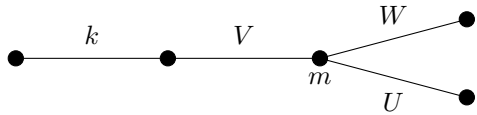
\begin{center}
	\begin{tikzpicture}
		\filldraw (-2,0) circle(0.1);
		\filldraw (0,0) circle(0.1);
		\filldraw (2,0) ++(-15:2) circle(0.1);
		\filldraw (2,0) circle(0.1) node[below=2pt] {\(m\)};
		\filldraw (2,0) ++(15:2) circle(0.1);
		\draw (-1.9,0) -- (-0.1,0) node[pos=0.5,above=2pt] {\(k\)};
		\draw (2,0) -- ++(-15:2) node[pos=0.5,below=2pt] {\(U\)};
		\draw (0.1,0) -- (1.9,0) node[pos=0.5,above=2pt] {\(V\)};
		\draw (2,0) -- ++(15:2) node[pos=0.5,above=2pt] {\(W\)};
	\end{tikzpicture}
\end{center}

From this, we see that the projective modules in the principal block are as follows: \(\PC(k) \sim [k \mid V \mid k]\), \(\PC(U) \sim [U \mid W \mid V \mid U \mid \ldots \mid V \mid U]\), \(\PC(W) \sim [W \mid V \mid U \mid W \mid \ldots \mid V \mid U \mid W]\) and \(\heart \PC(V) \coloneqq \rad(\PC(V))/\soc(\PC(V)) \cong k \oplus Y_V\) where \(Y_V \sim [U \mid W \mid V \mid \ldots \mid U \mid W]\).

\begin{propn} \label{SuzukiCase2}
	Let \(G = \Sz(q)\) and suppose that \(r\) is an odd prime dividing \(q - s + 1\) for \(s = \sqrt{2q}\). Then the value of \(\Ext_G^n(M,N)\), for irreducible \(M\), \(N\) in the principal \(r\)-block of \(G\), is nonzero for precisely the values of \(n\) modulo 8 given in the below table. Here, the entry in row \(M\), column \(N\) gives the values of \(n\) modulo 8 for which \(\Ext_G^n(M,N) \cong k\).
	\[
		\begin{array}{lcccc}
				&	k		&	U		&	V					&	W		\\ \midrule
			k	&	0, \, 7	&	2, \, 3	&	1, \, 6				&	4, \, 5	\\
			U	&	4, \, 5	&	0, \, 7	&	3, \, 6				&	1, \, 2	\\
			V	&	1, \, 6	&	1, \, 4	&	0, \, 2, \, 5, \, 7	&	3, \, 6	\\
			W	&	2, \, 3	&	5, \, 6 &	1, \, 4				&	0, \, 7 \\ \bottomrule
		\end{array}
	\]
	All non-projective irreducible modules \(M\) outside of the principal block (these are the modules with characters \(\Lambda_i\)) are such that \(\Ext_G^n(M, M) \cong k\) for all \(n\).
\end{propn}

The latter statement is a direct consequence of \cref{LonelyModule}. The remainder of the proof of \cref{SuzukiCase2} is given as the combination of the following two propositions.

\begin{propn} \label{SuzukiCase2Part1}
	The dimensions of \(\Ext_G^n(M,N)\) are as in the table in \cref{SuzukiCase2} for \((M,N) \in \{k, \, U, \, V, \, W\}^2\setminus\{(V,V)\}\).
\end{propn}

\begin{proof}
	Throughout, the reader should refer to the structure of the projective modules given before \cref{SuzukiCase2}.	We proceed in the usual way, examining the structure of \(\Omega^n k\). Let \(Y_U\) and \(Y_W\) denote \(\heart (\PC(U))\) and \(\heart (\PC(W))\), respectively, and note that \(\head Y_U \cong W\), \(\head Y_V \cong U\) and \(\head Y_W \cong V\). Where \(\head \Omega^n V\) is simple, the structure of \(\Omega^{n+1} V\) may be immediately read off from the shape of \(\PC(\Omega^n V)\).

	First, note that \(\Omega k\) has shape \([V \mid k]\) and thus \(\Omega^2 k\) must have shape \([Y_V \mid V]\). Then \(\Omega^3 k \cong U\) and so \(\Omega^4 k \cong \Omega U\) has shape \([Y_U \mid U]\). This then immediately gives \(\Omega^5 k \cong W\), leading to \(\Omega^6 k \cong \Omega W\) with shape \([Y_W \mid W]\). Finally, this gives \(\Omega^7 k\) of shape \([k \mid V]\) and thus \(\Omega^8 k \cong k\), so as in the previous case we see that \(k\), \(U\) and \(W\) are periodic of period 8. By examining the heads of these modules (and using the fact that \(\Ext_G^n(M,N) \cong \Ext_G^n(N^*, M^*)\)) we obtain the desired result.
\end{proof}

\begin{propn} \label{SuzukiCase2Part2}
	The dimensions of \(\Ext_G^n(V,V)\) are as in the table in \cref{SuzukiCase2}. In particular, \(\Ext_G^n(V,V) \cong k\) precisely when \(n \equiv 0\), 2, 5 or \(7 \mod 8\).
\end{propn}

\begin{proof}
	As with the previous case, we examine the structure of \(\Omega^n V\) while referring continually to the structure of the projective modules given before \cref{SuzukiCase2}. We first provide the shapes of \(\Omega^n V\) for \(n = 1,\) \ldots, \(8\).
	\[\arraycolsep=10pt
	\begin{array}{llll}
		\Omega V \sim \bigmoduleshape{k \oplus Y_V}{V}	&	\Omega^2 V \sim \bigmoduleshape{V}{k \oplus U}	&	\Omega^3 V \sim \quad \ \ Y_U						&	\Omega^4 V \sim \bigmoduleshape{U}{W}	\\
		\Omega^5 V \sim \quad \ \ Y_W					&	\Omega^6 V \sim \bigmoduleshape{k \oplus W}{V}	&	\Omega^7 V \sim \bigmoduleshape{Y_W}{k \oplus W}	&	\Omega^8 V \cong \quad V
	\end{array}
	\]
	The cases where \(\Omega^{n-1} V\) has a simple head may be read off directly from the structure of its projective cover. 

	For \(\Omega^2 V\), note that the \(V\) in \(\soc \Omega V\) must come from a diagonal submodule of \(\heart \perm \oplus \soc Y_U\) in \(\perm \oplus \PC(U)\). Similarly, for \(\Omega^7 V\), the \(V\) in \(\soc \Omega^6 V\) must come from a diagonal submodule of \(\heart \perm \oplus \head Y_W\) in \(\perm \oplus \PC(W)\). The result then follows by examining the heads of the above modules.
\end{proof}



Finally, we consider the case where \(r \mid q + s + 1\). In this case, the modules with characters \(\Omega_i\) and \(\Lambda_i\) are projective and the principal \(r\)-block of \(G\) contains 4 simple modules: \(k\), \(U\), \(V\) and \(W\) where \(U^* \cong W\), \(\dim U = \dim W = \frac{s}{2}(q-1)\) and \(\dim V = (q-1)(q-s+1)\). The remaining modules lie alone in blocks of maximal defect as before. From \cite[423]{BurkhardtSuzuki}, the principal block has the below Brauer tree with exceptionality \(m = \frac{r^x - 1}{4}\) where \(r^x\) is the \(r\)-part of \(q + s + 1\). Note that the case \(q = 2\) is not covered in \cite{BurkhardtSuzuki} but one may verify directly that the Brauer tree in this case is still as below with \(m = 1\).
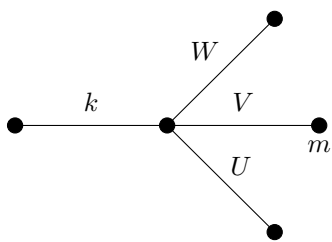
\begin{center}
	\begin{tikzpicture}
		\filldraw (-2,0) circle(0.1);
		\filldraw (0,0) circle(0.1);
		\filldraw (0,0) ++(45:2) circle(0.1);
		\filldraw (2,0) circle(0.1) node[below=2pt] {\(m\)};
		\filldraw (0,0) ++(-45:2) circle(0.1);
		\draw (-2,0) -- (0,0) node[pos=0.5,above=2pt] {\(k\)};
		\draw (0,0) -- ++(-45:2) node[pos=0.5,above=5pt,right=0pt] {\(U\)};
		\draw (0,0) -- (2,0) node[pos=0.5,above=2pt] {\(V\)};
		\draw (0,0) -- ++(45:2) node[pos=0.5,above=8pt,left=-3pt] {\(W\)};
	\end{tikzpicture}
\end{center}

This Brauer tree is a star, so the below result then follows from \cref{StarCohomologyExceptionalOuter,LonelyModule} (and \cref{StarCohomologyExceptionalMiddle} when \(m = 1\)).

%

\begin{propn} \label{SuzukiCase3}
	Let \(G = \Sz(q)\) and suppose that \(r\) is an odd prime dividing \(q + s + 1\) for \(s = \sqrt{2q}\). Then, provided \(m \neq 1\) ({\itshape i.e.} the \(r\)-part of \(q + s + 1\) is not \(5\)) the value of \(\Ext_G^n(M,N)\), for irreducible \(M\), \(N\) in the principal \(r\)-block of \(G\), is nonzero for precisely the values of \(n\) modulo 8 given in the below table. Here, the entry in row \(M\), column \(N\) gives the values of \(n\) modulo 8 for which \(\Ext_G^n(M,N) \cong k\).
	\[
	\begin{array}{lcccc}
				&	k			&	U			&	V			&	W			\\ \midrule
			k	&	0, \, 7		&	1, \, 2		& 	3, \, 4		&	5, \, 6 	\\
			U	&	5, \, 6		&	0, \, 7		&	1, \, 2		&	3, \, 4		\\
			V	&	3, \, 4		&	5, \, 6		&	\text{all}	&	1, \, 2		\\
			W	&	1, \, 2		&	3, \, 4		&	5, \, 6		&	0, \, 7		\\	\bottomrule
	\end{array}
	\]
	When the \(r\)-part of \(q + s + 1\) is \(5\), we instead have \(\Ext_G^n(V,V) \cong k\) for \(n \equiv 3\), \(4 \mod 8\) and zero otherwise. Further, all non-projective irreducible \(kG\)-modules \(M\) outside the principal block (so \(M\) has character \(\Theta_i\) for \((q + s + 1)_{r'} \nmid i\)) are such that \(\Ext_G^n(M, M) \cong k\) for all \(n\).
\end{propn}
\section{Extensions and cohomology in \texorpdfstring{\(\Ree(q)\)}{Ree groups}} \label{sec:Ree}

Finally, we deal with the Ree groups \(G \coloneqq \Ree(q)\) for \(q = 3^{2a+1}\) and \(a \geq 0\) (note here we mean that for \(a > 0\), \(G\) is a finite simple group and for \(a = 0\) we have \(G \cong \PGammaL_2(8)\)). Provided \(r > 3\), all of the Sylow \(r\)-subgroups of \(G\) are cyclic and the Brauer trees for these cases may be found in \cite[\textsection 4.1]{HissReeBrauerTrees}. Note that the Sylow \(2\)-subgroups of \(G\) are elementary abelian of order eight, and so the representation type in this case is wild and we do not consider this case here.

As in the case of the Suzuki groups, we will not require any structural information about \(G\) in this case as the results are dependent solely upon the Brauer trees of the blocks involved. These groups have order \(\abs{G} = q^3 (q -1) (q^3 + 1)\) which factors as \(q^3 (q - 1) (q + 1) (q + s + 1) (q - s - 1)\) where \(s = \sqrt{3q}\). As such, the study of the cross-characteristic representation theory of these groups splits naturally into the cases where \(r\) divides \(q \pm 1\) or \(q \pm s + 1\) (noting that since \(r > 3\) it may divide at most one such factor).

We first consider the case where \(r \mid q - 1\) (and so \(a > 0\)). From \cite[Theorem 4.1]{HissReeBrauerTrees}, the blocks of maximal defect in this case have at most two irreducible \(kG\)-modules, and the Brauer trees of the two blocks with more than one irreducible module are lines with exceptional vertex in the middle.

\begin{propn} \label{ReeMinusOne}
	Let \(G = \Ree(q)\) and suppose that \(r\) is an odd prime divisor of \(q - 1\). Then the only two blocks with more than one irreducible module contain only two, \(S_1\) and \(S_2\), such that, for \(i \neq j\),
	\[\Ext_G^n(S_i,S_i) \cong \begin{cases}
		0	&	n \equiv 1, \ 2 \mod 4,\\
		k	&	n \equiv 0, \ 3 \mod 4,
	\end{cases} \qquad 
	\Ext_G^n(S_i,S_j) \cong \begin{cases}
		0	&	n \equiv 0, \ 3 \mod 4,\\
		k	& 	n \equiv 1, \ 2 \mod 4.
	\end{cases}\]
	All other non-projective irreducible modules \(S\) lie alone in their blocks and so \(\Ext_G^n(S, S) \cong k\) for all \(n \geq 0\).
\end{propn}

\begin{proof}
	The Brauer trees in this case are given in \cite[Theorem 4.1]{HissReeBrauerTrees}. The \(\Ext\)s for blocks with these trees were previously worked out in \cite[Proposition 3.6]{Paper2}, but this also follows from \cref{LineCohomologyExceptionalInner}. For all modules lying in blocks alone, we use \cref{LonelyModule}.
\end{proof}

Next we suppose \(2 < r \mid q + 1\) (and so \(a > 0\)). From \cite[Theorem 4.2]{HissReeBrauerTrees}, the principal block has the below Brauer tree and there is one other block of maximal defect containing two irreducible \(kG\)-modules with Brauer tree a line with exceptional vertex on the outside.
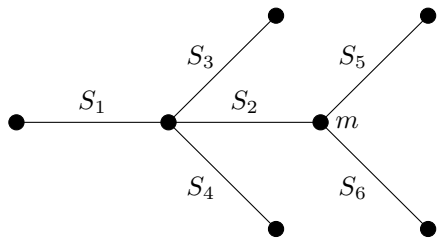
\begin{center}
		\begin{tikzpicture}
			\filldraw (-2, 0) circle(0.1);
			\filldraw (0,0) circle(0.1);
			\filldraw (0,0) ++(45:2) circle(0.1);
			\filldraw (0,0) ++(-45:2) circle(0.1);
			\filldraw (2,0) circle(0.1) node[right=2pt] {\(m\)};
			\filldraw (2,0) ++(45:2) circle(0.1);
			\filldraw (2,0) ++(-45:2) circle(0.1);
			\draw (-2,0) -- (0,0) 			node[pos=0.5,above=0pt] {\(S_1\)};
			\draw (0,0) -- (2,0) 			node[pos=0.5,above=0pt] {\(S_2\)};
			\draw (0,0) -- ++(45:2)		 	node[pos=0.3, above=5pt] {\(S_3\)};
			\draw (0,0) -- ++(-45:2)	 	node[pos=0.3, below=5pt] {\(S_4\)};
			\draw (2,0) -- ++(45:2)		 	node[pos=0.3, above=5pt] {\(S_5\)};
			\draw (2,0) -- ++(-45:2)	 	node[pos=0.3, below=5pt] {\(S_6\)};
		\end{tikzpicture}
\end{center}
In this case, the projective covers of all irreducible modules bar \(S_2\) are uniserial. The heart of \(\PC(S_2)\) is of shape \([S_4 \mid S_1 \mid S_3] \oplus Y_2\) where \(Y_2 \sim [S_5 \mid S_6 \mid S_2 \mid \cdots \mid S_6]\) and we denote the hearts of \(\PC(S_5)\) and \(\PC(S_6)\) by \(Y_5\) and \(Y_6\), respectively, where \(Y_5 \sim [S_6 \mid S_2 \mid S_5 \mid \cdots \mid S_2]\) and \(Y_6 \sim [S_2 \mid S_5 \mid S_6 \mid \cdots \mid S_5]\).

\begin{propn} \label{forkylad}
	Let \(G = \Ree(q)\) and suppose that \(r\) is an odd prime divisor of \(q+1\). Then the \(\Ext\)s for the principal block may be found in the below table, where the entries in row \(S_i\), column \(S_j\) give those values of \(n\) mod 12 for which \(\dim \Ext(S_i, S_j) = 1\). For all other values of \(n\) mod 12, this \(\Ext\) is zero. In this case, \(S_1\) is the trivial module.
	\[
	\begin{array}{lcccccc}
				&	S_1			&	S_2												&	S_3			&	S_4			&	S_5			&	S_6			\\ \midrule
			S_1	&	0, \, 11	&	3, \, 8											& 	9, \, 10	&	1, \, 2 	&	6, \, 7		&	4, \, 5		\\
			S_2	&	3, \, 8		&	0, \, 2, \, 4, \, 5, \, 6, \, 7, \, 9, \, 11	&	1, \, 6		&	5, \, 10	&	3, \, 10	&	1, \, 8		\\
			S_3	&	1, \, 2		&	5, \, 10										&	0, \, 11	&	3, \, 4		&	4, \, 5		&	6, \, 7		\\
			S_4	&	9, \, 10	&	1, \, 6											&	7, \, 8		&	0, \, 11	&	4, \, 5		&	2, \, 3		\\
			S_5	&	4, \, 5		&	1, \, 8											&	2, \, 3		&	6, \, 7		&	0, \, 11	&	9, \, 10	\\
			S_6	&	6, \, 7		&	3, \, 10										&	4, \, 5		&	8, \, 9		&	1, \, 2		&	0, \, 11	\\	\bottomrule
	\end{array}
	\]
	There is only one other block containing more than one irreducible module, which contains only two irreducible modules \(T_1\), \(T_2\), such that \(\Ext_G^n(T_2,T_2) \cong k\) for all \(n\), and, for \(i \neq j\),
	\[\Ext_G^n(T_1,T_1) \cong \begin{cases}
		0	&	n \equiv 1, \ 2 \mod 4,\\
		k	&	n \equiv 0, \ 3 \mod 4,
	\end{cases} \qquad 
	\Ext_G^n(T_i,T_j) \cong \begin{cases}
		0	&	n \equiv 0, \ 3 \mod 4,\\
		k	&	n \equiv 1, \ 2 \mod 4.
	\end{cases}\]
	Finally, all other non-projective irreducible modules \(S\) lie alone in their blocks and so \(\Ext_G^n(S, S) \cong k\) for all \(n \geq 0\).
\end{propn}

\begin{proof}
	For the non-principal blocks with only one irreducible module, either this irreducible module is projective or we are done by \cref{LonelyModule}. For the non-principal block with two irreducible modules, this was done in \cite[Proposition 3.5]{Paper2} but also follows from \cref{LineCohomologyExceptionalOuter}. The bulk of the work in this case is used to determine the \(\Ext\)s for the principal block, which we do now. Since all modules involved are uniserial, it is easy to calculate \(\Omega^n S_1\) for \(1 \leq n \leq 12\) and observe the following.
	\[
	\begin{array}{llllll}
		\Omega S_1 \sim \fourmoduleshape{S_4}{S_2}{S_3}{S_1}	&	\Omega^2 S_1 \sim S_4	&	\Omega^3 S_1 \sim \fourmoduleshape{S_2}{S_3}{S_1}{S_4}	&	\Omega^4 S_1 \sim \twomoduleshape{Y_2}{S_2}	&	\Omega^5 S_1 \sim S_6	&	\Omega^6 S_1 \sim \twomoduleshape{Y_6}{S_6}	\\[25pt]
		\Omega^7 S_1 \sim S_5	&	\Omega^8 S_1 \sim \twomoduleshape{Y_5}{S_5}	&	\Omega^9 S_1 \sim \fourmoduleshape{S_3}{S_1}{S_4}{S_2}	&	\Omega^{10} S_1 \sim S_3	&	\Omega^{11} S_1 \sim \fourmoduleshape{S_1}{S_4}{S_2}{S_3}	&	\Omega^{12} S_1 \sim S_1
	\end{array}
	\]
	Examining the above fills all rows of the table bar the second, for which we need to also calculate \(\Omega^n S_2\) for \(1 \leq n \leq 12\). This is somewhat more involved since the modules which appear are not all uniserial. The first few are still straightforward:
	\[
	\begin{array}{lll}
		\Omega S_2 \sim \twomoduleshape{\threemoduleshape{S_3}{S_1}{S_4} \oplus Y_2}{\hphantom{S_2}S_2}					&	\Omega^2 S_2 \sim \twomoduleshape{S_2}{S_3 \oplus S_6}			&	\Omega^3 S_2 \sim \twomoduleshape{\twomoduleshape{S_1}{S_4} \oplus \rad Y_2}{S_2}	\\[25pt]
		\Omega^4 S_2 \sim \twomoduleshape{S_2}{\twomoduleshape{S_3}{S_1} \oplus \twomoduleshape{S_6}{S_5}}	&	\Omega^5 S_2 \sim \twomoduleshape{S_4 \oplus \rad^2 Y_2}{S_2\hphantom{\rad^2}}	&	
	\end{array}
	\]
	To compute \(\Omega^6 S_2\), take the submodule of \(\PC(S_4) \oplus \PC(S_2)\) whose quotient is \(S_4 \oplus \rad^2 Y_2\). Then \(\Omega^6 S_2\) is the preimage in this submodule of a diagonal submodule of its head, which is of the below shape.
	
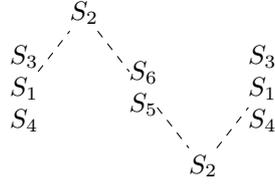
\begin{center}
		\begin{tikzpicture}
			\matrix(A)[matrix of math nodes, nodes in empty cells] {
												&	S_{2}	&								&		&										\\
			\threemoduleshape{S_3}{S_1}{S_4}	&			&	\twomoduleshape{S_6}{S_5}	&		&	\threemoduleshape{S_3}{S_1}{S_4}	\\
												&			&								&	S_2	&										\\
			};

			\draw[dashed, shorten <>= 0.3cm] (A-2-1.center) -- (A-1-2.center);
			\draw[dashed, shorten <>= 0.3cm] (A-1-2.center) -- (A-2-3.center);
			\draw[dashed, shorten <>= 0.3cm] (A-2-3.center) -- (A-3-4.center);
			\draw[dashed, shorten <>= 0.3cm] (A-3-4.center) -- (A-2-5.center);
		\end{tikzpicture}
	\end{center}
	The remaining Heller translates are then again relatively straightforward.
	\[
	\begin{array}{lll}
		\Omega^7 S_2 \sim \twomoduleshape{\hphantom{\rad^3 S_2}S_2}{\twomoduleshape{\rad^3 Y_2}{S_2} \oplus S_3}	&	\Omega^8 S_2 \sim \twomoduleshape{\twomoduleshape{S_1}{S_4} \oplus \twomoduleshape{S_6}{S_5}}{S_2}					&	\Omega^9 S_2 \sim \twomoduleshape{S_2 \hphantom{\rad^3}}{\twomoduleshape{S_3}{S_1} \oplus \twomoduleshape{\rad^2 Y_6}{S_6}}	\\[25pt]
		\Omega^{10} S_2 \sim \twomoduleshape{S_4 \oplus S_5}{S_2}							&	\Omega^{11} S_2 \sim \twomoduleshape{S_2 \hphantom{S_2}}{\threemoduleshape{S_3}{S_1}{S_4} \oplus \twomoduleshape{\rad Y_5}{S_5}}	&	\Omega^{12} S_2 \cong S_2
	\end{array}
	\]
	The result then follows from examining the heads of the given modules.
\end{proof}

We next suppose that \(2 < r \mid q + s + 1\). Then from \cite[Theorem 4.3]{HissReeBrauerTrees}, the only block containing more than one irreducible \(kG\)-module is the principal block which has the below Brauer tree.
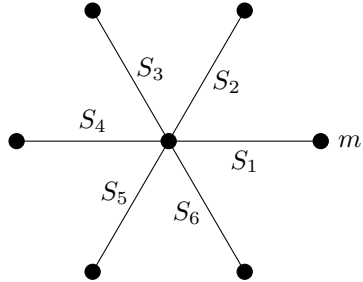
\begin{center}
	\begin{tikzpicture}
		\coordinate (O) at (0,0);

		\filldraw (O) 			circle(0.1);
		\filldraw (O) ++(0:2) 	circle(0.1) node[right=3pt] {\(m\)};
		\filldraw (O) ++(60:2) 	circle(0.1);
		\filldraw (O) ++(120:2) circle(0.1);
		\filldraw (O) ++(180:2) circle(0.1);
		\filldraw (O) ++(240:2) circle(0.1);
		\filldraw (O) ++(300:2) circle(0.1);

		\draw (O) -- ++(0:2) 	node[pos=0.5, below=0pt] 				{\(S_1\)};
		\draw (O) -- ++(60:2) 	node[pos=0.45, below=0pt, right=0pt] 	{\(S_2\)};
		\draw (O) -- ++(120:2) 	node[pos=0.55, above=0pt, right=0pt] 	{\(S_3\)};
		\draw (O) -- ++(180:2) 	node[pos=0.5, above=0pt] 				{\(S_4\)};
		\draw (O) -- ++(240:2) 	node[pos=0.4, above=0pt, left=0pt] 		{\(S_5\)};
		\draw (O) -- ++(300:2) 	node[pos=0.55, below=0pt, left=0pt] 	{\(S_6\)};
	\end{tikzpicture}
\end{center}

The following is then immediate from \cref{StarCohomologyExceptionalOuter} and \cref{LonelyModule}.

\begin{propn} \label{ReeStar}
	Let \(G = \Ree(q)\) and suppose that \(r\) is an odd prime divisor of \(q + s + 1\) where \(s = \sqrt{3q}\). Then the \(\Ext\)s for the principal block may be found in the below table, where the entries in row \(S_i\), column \(S_j\) give those values of \(n\) mod 12 for which \(\dim \Ext_G^n(S_i, S_j) = 1\). For all other values of \(n\) mod 12, this \(\Ext\) is zero. To avoid conflicts with the notation used in \cref{StarCohomologyExceptionalOuter}, \(S_4\) is the trivial module rather than \(S_1\).
	\[
	\begin{array}{lcccccc}	\toprule
				&	S_1			&	S_2			&	S_3			&	S_4			&	S_5			&	S_6			\\ \midrule
			S_1	&	\text{all}	&	9, \, 10	& 	7, \, 8		&	5, \, 6 	&	3, \, 4		&	1, \, 2		\\
			S_2	&	1, \, 2		&	0, \, 11 	&	9, \, 10	&	7, \, 8		&	5, \, 6		&	3, \, 4		\\
			S_3	&	3, \, 4		&	1, \, 2		&	0, \, 11	&	9, \, 10	&	7, \, 8		&	5, \, 6		\\
			S_4	&	5, \, 6		&	3, \, 4		&	1, \, 2		&	0, \, 11	&	9, \, 10	&	7, \, 8		\\
			S_5	&	7, \, 8		&	5, \, 6		&	3, \, 4		&	1, \, 2		&	0, \, 11	&	9, \, 10	\\
			S_6	&	9, \, 10	&	7, \, 8		&	5, \, 6		&	3, \, 4		&	1, \, 2		&	0, \, 11	\\	\bottomrule
	\end{array}
	\]
	All other non-projective irreducible modules \(S\) lie alone in their blocks and so \(\Ext_G^n(S, S) \cong k\) for all \(n \geq 0\).
\end{propn}

Finally, we suppose that \(2 < r \mid q - s + 1\) (and so \(a > 0\)). In this case, the principal block is again the only one with more than one irreducible \(kG\)-module and has the below Brauer tree. The planar embedding of this tree is not determined in \cite[Theorem 4.4]{HissReeBrauerTrees}.

\begin{center}
		\begin{tikzpicture}
			\filldraw (-2,0) 			circle(0.1);
			\filldraw (0,0) 			circle(0.1);
			\filldraw (2,0) 			circle(0.1) node[right=2pt] {\(m\)};
			\filldraw (2,0) ++(45:2)	circle(0.1);
			\filldraw (2,0) ++(135:2)	circle(0.1);
			\filldraw (2,0) ++(-135:2) 	circle(0.1);
			\filldraw (2,0) ++(-45:2)	circle(0.1);
			\draw (-2,0) -- (0,0) 		node[pos=0.5, above=0pt] {\(S_1\)};
			\draw (0,0) -- (2,0) 		node[pos=0.5, above=0pt] {\(S_2\)};
			\draw (2,0) -- ++(-135:2)	node[pos=0.3, below=5pt] {\(S_3\)};
			\draw (2,0) -- ++(-45:2)	node[pos=0.3, below=5pt] {\(S_4\)};
			\draw (2,0) -- ++(45:2)		node[pos=0.3, above=5pt] {\(S_5\)};
			\draw (2,0) -- ++(135:2)	node[pos=0.3, above=5pt] {\(S_6\)};
		\end{tikzpicture}
\end{center}
In this case, the projective covers of all modules bar \(S_2\) are uniserial. We have \(\PC(S_1) \sim [S_1 \mid S_2 \mid S_1]\) and for \(3 \leq i \leq 6\) we denote the heart of \(\PC(S_i)\) by \(Y_i\) and let \(Y_2\) be such that the heart of \(\PC(S_2)\) is \(S_1 \oplus Y_2\), where \(Y_2 \sim [S_6 \mid S_5 \mid S_4 \mid S_3 \mid S_2 \mid \cdots \mid S_3]\) and \(Y_{i+1}\) is obtained by applying the permutation \(\sigma \coloneqq (2,3,4,5,6)\) to the indices of the factors of \(Y_i\) (so, the socle of \(Y_3\) is \(S_{\sigma(6)} = S_2\), and so on).

\begin{propn} \label{ReeLongStar}
	Let \(G = \Ree(q)\) and suppose that \(r\) is an odd prime divisor of \(q - s + 1\) where \(s = \sqrt{3q}\). Then the \(\Ext\)s for the principal block may be found in the below table, where the entries in row \(S_i\), column \(S_j\) give those values of \(n\) mod 12 for which \(\dim \Ext(S_i, S_j) = 1\). For all other values of \(n\) mod 12, this \(\Ext\) is zero. In this case, \(S_1\) is the trivial module but it was not determined in \cite{HissReeBrauerTrees} precisely which simple modules \(S_3\)--\(S_6\) are.
	\[
	\begin{array}{lcccccc}	\toprule
				&	S_1			&	S_2						&	S_3			&	S_4			&	S_5			&	S_6			\\ \midrule
			S_1	&	0, \, 11	&	1, \, 10				& 	2, \, 3		&	4, \, 5 	&	6, \, 7		&	8, \, 9		\\
			S_2	&	1, \, 10	&	0, \, 2, \, 9, \, 11	&	1, \, 4		&	3, \, 6		&	5, \, 8		&	7, \, 10	\\
			S_3	&	8, \, 9		&	7, \, 10				&	0, \, 11	&	1, \, 2		&	3, \, 4		&	5, \, 6		\\
			S_4	&	6, \, 7		&	5, \, 8					&	9, \, 10	&	0, \, 11	&	1, \, 2		&	3, \, 4		\\
			S_5	&	4, \, 5		&	3, \, 6					&	7, \, 8		&	9, \, 10	&	0, \, 11	&	1, \, 2		\\
			S_6	&	2, \, 3		&	1, \, 4					&	5, \, 6		&	7, \, 8		&	9, \, 10	&	0, \, 11	\\	\bottomrule
	\end{array}
	\]
	All non-projective irreducible modules \(S\) outside the principal block lie in blocks on their own and so \(\Ext_G^n(S, S) \cong k\) for all \(n \geq 0\).
\end{propn}

\begin{proof}
	For all blocks other than the principal block, we again use \cref{LonelyModule}. We approach the principal block as usual. First note that \(\Omega^n S_1\) is as below for \(1 \leq n \leq 12\).
	\[
	\begin{array}{llllll}
		\Omega S_1 \sim \twomoduleshape{S_2}{S_1}		&	\Omega^2 S_1 \sim \twomoduleshape{Y_2}{S_2}	&	\Omega^3 S_1 \cong S_3	&	\Omega^4 S_1 \sim \twomoduleshape{Y_3}{S_3}		&	\Omega^5 S_1 \cong S_4	
		&	\Omega^6 S_1 \sim \twomoduleshape{Y_4}{S_4}		\\[15pt]
		\Omega^7 S_1 \cong S_5							&	\Omega^8 S_1 \sim \twomoduleshape{Y_5}{S_5}	&	\Omega^9 S_1 \cong S_6	&	\Omega^{10} S_1 \sim \twomoduleshape{Y_6}{S_6}	&	\Omega^{11} S_1 \sim \twomoduleshape{S_1}{S_2}
		&	\Omega^{12} S_1 \cong S_1
	\end{array}
	\]
	As in the other cases, we need only calculate \(\Omega^n S_2\) to complete the proof. In this case, this is relatively straightforward and these Heller translates are given below. The result then follows as usual by noting which irreducible modules lie in the head of the above and below modules.
	\[
	\begin{array}{llll}
		\Omega S_2 \sim \twomoduleshape{S_1 \oplus Y_2}{S_2}	&	\Omega^2 S_2 \sim \twomoduleshape{S_2}{S_1 \oplus S_3}		&	\Omega^3 S_2 \sim Y_3						&	\Omega^4 S_2 \sim \twomoduleshape{S_3}{S_4}	\\[15pt]
		\Omega^5 S_2 \sim Y_4									&	\Omega^6 S_2 \sim \twomoduleshape{S_4}{S_5}					&	\Omega^7 S_2 \sim Y_5						&	\Omega^8 S_2 \sim \twomoduleshape{S_5}{S_6}	\\[15pt]
		\Omega^9 S_2 \sim Y_6									&	\Omega^{10} S_2 \sim \twomoduleshape{S_1 \oplus S_6}{S_2}	&	\Omega^{11} S_2 \sim \twomoduleshape{S_2}{S_1 \oplus Y_2}	&	\Omega^{12} S_2 \cong S_2
	\end{array}
	\]
\end{proof}

\printbibliography

\end{document}